\documentclass[12pt,twoside,reqno]{amsart}

\usepackage{microtype}
\usepackage{cite}
\usepackage[OT1]{fontenc}
\usepackage{type1cm}
\usepackage{amssymb}
\usepackage{dsfont}
\usepackage{comment}
\usepackage{xcolor}

\usepackage{geometry}
\usepackage{hyperref}
\hypersetup{
  colorlinks=true,
  linkcolor=black,
  anchorcolor=black,
  citecolor=black,
  filecolor=black,      
  menucolor=red,
  runcolor=black,
  urlcolor=black,
}

\numberwithin{equation}{section}

\theoremstyle{plain}
\newtheorem{theorem}{Theorem}[section]

\newtheorem{proposition}[theorem]{Proposition}

\newtheorem{lemma}[theorem]{Lemma}

\theoremstyle{remark}
\newtheorem{remark}[theorem]{Remark}

\theoremstyle{definition}
\newtheorem{definition}[theorem]{Definition}

\newcommand{\R}{\mathbb{R}}

\newcommand{\C}{\mathbb{C}}

\newcommand{\N}{\mathbb{N}}

\renewcommand{\ge}{\geqslant}

\renewcommand{\geq}{\geqslant}
\renewcommand{\leq}{\leqslant}

\DeclareMathOperator{\udimm}{\overline{dim}_M}

\DeclareMathOperator{\dimh}{dim_H}

\DeclareMathOperator{\diam}{diam}

\newcommand{\old}[1]{}

\author[P. Nissinen]{Petteri Nissinen}
\address{Department of Physics and Mathematics, University of Eastern Finland, P.O. Box 111, 80101 Joensuu, Finland}
\email{petteri.nissinen@uef.fi}

\author[I. Prause]{Istv\'an Prause}
\address{Department of Physics and Mathematics, University of Eastern Finland, P.O. Box 111, 80101 Joensuu, Finland}
\email{istvan.prause@uef.fi}

\thanks{The work was supported by the Finnish Academy project 355839.}

\title[Quasiconformal dimension distortion on a line]{On quasiconformal dimension distortion for subsets of the real line}

\begin{document}

\begin{abstract}
Optimal quasiconformal dimension distortions bounds for subsets of the complex plane have been established by Astala. We show that these estimates can be improved when one considers subsets of the real line of arbitrary Hausdorff dimension. We present some explicit numerical bounds. 
\end{abstract}

\maketitle

\smallskip



\section{Introduction}

Let $f: \C \to \C$ be a $K$-quasiconformal mapping. A classical result is that, such maps are H\"older continuous with exponent $1/K$, see e.g. \cite[Theorem 3.10.2]{AIM}. Much deeper lies their distortion properties in terms of Hausdorff dimension. These have been proved by Astala in \cite{Astala}. Namely, if $E \subset \C$ such that $\dimh(E)>0$, then he showed the dimension distortion bounds
\begin{align}
\label{Astalasestimate}
\frac{1}{\frac{1}{K}\left(\frac{1}{\dimh(E)}-\frac{1}{2}\right)+\frac{1}{2}}\geq \dimh(f(E))\geq \frac{1}{K\left(\frac{1}{\dimh(E)}-\frac{1}{2}\right)+\frac{1}{2}}.
\end{align}
Ransford obtained the same kind of bounds in the context of variation of Julia sets in a holomorphic motion in \cite{Ransford}. Other properties of quasiconformal maps, such as the higher Sobolev regularity and area distortion are also closely related, see \cite{Astala}. 

Astala also showed in \cite{Astala} that the bounds are sharp in \eqref{Astalasestimate}. See also Theorem 1.7 of \cite{FRY} for a recent related result in the setting of holomorphic motions. These examples for sharpness are based on Cantor sets spread out in the plane and it is natural to expect that e.g. for subsets of a line the distortion is more constrained. 

Indeed, if one asks about dimension bounds of the entire $f(\R)$, Smirnov proved that the dimension of a $K$-quasicircle is at most $1+k^2$ in \cite{Smi}, which Ivrii further improved in \cite{Ivrii}. Here and throughout the paper $0\leq k<1$ and $K \ge 1$ always satisfy the relation $k=\frac{K-1}{K+1}$. Dimension distortion bounds for \emph{subsets} $E \subset \R$ have been studied, for instance, in \cite{ACMOU,HPS,Pr07,PrSm}. In these works, either only $1$-dimensional sets are considered (or at least the dimension being close to one) or as in \cite{PrSm} it is assumed that the map preserves the real line, $f(\R)=\R$.

In this paper, we consider quasiconformal maps $f \colon \C \to \C$ with no further symmetry assumptions and subsets $E \subset \R$ of \emph{arbitrary} Hausdorff dimension. Our goal is to show that in this setting, that is under the assumption $E \subset \R$, the estimates in \eqref{Astalasestimate} are not sharp, i.e. can be improved upon. We establish explicit bounds in Theorem \ref{smallerthanone1} (lower bound) and Theorem \ref{smallerthanone2} (upper bound) below. Note that in the special case $\dimh (E)=1$ we have the clean bounds of \cite[Theorem 3]{HPS} and \cite[Theorem 1]{Smi}
\[ 1-k^2 \leq \dimh (f(E)) \leq 1+k^2.\]
Unfortunately, the bounds in the general case $0<\dimh (E)<1$ in Theorems \ref{smallerthanone1} and \ref{smallerthanone2} are not as clean and provide only a weaker improvement.

Our method is based on decomposition properties of quasiconformal maps, this leads us to consider antisymmetric maps and symmetric ones. Ideas of \cite{Smi} leads to dimension distortion bounds in the antisymmetric case. We treat this case in detail in Section \ref{se:antisymmetric}.
Fuhrer, Ransford and Younsi very recently streamlined and unified quasiconformal dimension distortion results via inf-harmonic functions \cite{FRY}. We are able to make use of their framework in certain parts of the argument, see Theorem \ref{upward} and Theorem \ref{thm:Minkowskilower}. We prove the lower bound for Hausdorff dimension distortion in the antisymmetric case in Theorem \ref{downward}. The approach here is based on \cite[Ch. 13]{AIM}. 
In the symmetric case, we can directly use results of \cite{PrSm}. The two type of estimates are combined and further amplified via yet another decomposition in Section \ref{se:improved}. Necessary background and definitions are recalled in Section \ref{se:pre}.

\section{Preliminaries} \label{se:pre}
In this paper we are only concerned with subsets $E$ of the complex plane $\C$. We refer to the standard definitions of Hausdorff outer measure and \emph{Hausdorff and Minkowski dimension}, for example to the book \cite[pp. 19-20, p.25,  pp. 28-29]{Falconer0}.
It will be convenient for us to consider an equivalent definition for Hausdorff dimension (denoted below by $\dimh$) in terms of covering with balls which we review next briefly. We take the following infimum 
\begin{align}
{\mathcal{B}}_{\delta}^s(E)=\inf\left\{\sum_{i=1}^{\infty} {\left(\diam B_i\right)}^s: \{B_i\} \ \textrm{is a} \ \delta\textrm{-cover of $E$ by balls}\right\}.
\end{align}
We use the notation $\diam$ for diameter of a set and by $\delta$-cover we require that $\diam B_i \leqslant \delta$.
We obtain an outer measure 
\begin{align}
{\mathcal{B}}^s(E)=\lim_{\delta \downarrow 0}{\mathcal{B}}_{\delta}^s(E)=\sup_{\delta>0} {\mathcal{B}}_{\delta}^s(E).
\end{align}

By \cite[pp. 32-33]{Falconer0}, we have the following comparison between Hausdorff outer measure $\mathcal{H}^s$ and outer measure $\mathcal{B}^s$
\begin{align}
\label{pallopeitteetvrtkaikki peitteet}
\mathcal{H}^s(E)\leq {\mathcal{B}}^s(E)\leq 2^s\mathcal{H}^s(E).
\end{align}
In particular, this implies that the values of $s$ at which $\mathcal{H}^s$ and ${\mathcal{B}}^s$ jumps from $\infty$ to $0$ are equal, so that the dimensions defined by the two outer measures are equal.

\begin{definition}
A homeomorphism $f:{\C} \to {\C}$ is \emph{$K$-quasiconformal} with $K\geq 1$ if it is orientation preserving, $f\in W_{loc}^{1, 2}({\C})$ and for almost every $z\in {\C}$ the directional derivatives satisfy 
\begin{align*}
\max_{\alpha} \left|\partial_{\alpha} f(z)\right|\leq K \min_{\alpha} \left|\partial_{\alpha} f(z)\right|.
\end{align*}
\end{definition}

Equivalently, we could define quasiconformal mappings via the Beltrami equation: a homeomorphism $f:{\C} \to {\C}$ is {$K$-quasiconformal} if it lies  in the Sobolev space $W_{loc}^{1, 2}({\C})$ and satisfies the Beltrami equation
\begin{align}
\label{Beltramiequation}
\frac{\partial f}{\partial\overline{z}}(z)=\mu(z)\frac{\partial f}{\partial z}(z) \ \ \textrm{for almost all} \ \ z\in{\C}
\end{align}
for some measurable function $\mu:{\C}\to {\C}$ with $\left\|\mu\right\|_{\infty}\leq k <1$, where $k=(K-1)/(K+1)$  and $\left\|\mu\right\|_{\infty}$ is essential supremum. We call $\mu=\mu_f$ the Beltrami coefficient of $f$. We refer to the book \cite{AIM} for a detailed account on the theory of quasiconformal mappings, especially chapter 13 is relevant for the questions discussed in this paper.

Next, let us consider some special classes of quasiconformal maps. If the Beltrami coefficient is compactly supported an often convenient normalization is the following.

\begin{definition}
A quasiconformal mapping $f \colon \C \to \C$ is called a \emph{principal mapping} if its Beltrami coefficient is compactly supported and satisfies $f(z)=z+O\left(\frac{1}{z}\right)$ as $z \to \infty$.
\end{definition}

\begin{definition}
A quasiconformal mapping $f:{\C} \to {\C}$ is \emph{antisymmetric} if its Beltrami coefficient $\mu_f$ satisfies
\begin{align*}
\mu_f(z)=-\overline{\mu_f(\overline{z})}
\end{align*}
for almost all $z\in {\C}$. 
\end{definition}

This definition is motivated by the work \cite{Smi}.
Conversely, we call a quasiconformal mapping $f:{\C} \to {\C}$ \emph{symmetric} if $f(\overline{z})=\overline{f(z)}$ for all $z\in{\C}$. In this case, $f$ preserves real line (i.e. $f({\R})={\R}$) and $\mu_f(\overline{z})=\overline{\mu_f(z)}$ for almost all $z\in{\C}$. We remark that in \cite{PrSm} such maps were called $k$-quasisymmetric.

The relevance of the special classes above is that any quasiconformal map can be decomposed to a symmetric map followed by an antisymmetric one \cite{Smi}. We recall the precise statement in Section \ref{se:improved} below. Related decomposition ideas go back to \cite{kuhnau}.

One key idea to understand dimension distortion of quasiconformal maps in \cite{Astala} was the connection to holomorphic motions. We recall the definition.
\begin{definition}
Let $A\subset {\C}$. A \emph{holomorphic motion} of $A$ is a map $\Phi: \mathbb{D} \times A \to {\C}$ such that 
\begin{itemize}
	\item For any fixed $a\in A$,  the map $\lambda \mapsto \Phi(\lambda, a)$ is holomorphic in $\mathbb{D}$. \\
	\item For any fixed $\lambda\in \mathbb{D}$, the map $a \mapsto \Phi(\lambda, a)=\Phi_{\lambda}(a)$ is an injection.\\
	\item The mapping $\Phi_0$ is the identity on $A$, i. e. $\Phi(0, a)=a$ for every $a\in A$. 
\end{itemize}
\end{definition}

Any quasiconformal map can be embedded into a holomorphic motion captured by the following theorem.
\begin{theorem}[Theorem 12.5.3 \cite{AIM}]
\label{thm:embedded}
Let $f:{\C} \to {\C}$ be $K=\frac{1+k}{1-k}$-quasiconformal. Then there exists a holomorphic motion $\Phi: \mathbb{D}\times {\C} \to {\C}$ such that $f(z)=\Phi(k, z)=\Phi_k(z)$, $z\in {\C}$.
\end{theorem}

Finally, we recall some relevant definitions and results about harmonic functions from \cite{FRY}.
\begin{definition}
A harmonic function $h:\mathbb{D} \to {\R}$ is called \emph{symmetric} if $h(\overline{\lambda})=h(\lambda)$ for all $\lambda\in \mathbb{D}$. A function $u \colon \mathbb{D} \to [0, \infty)$ is \emph{inf-sym-harmonic} if there is a family $\mathcal{H}$ of symmetric harmonic functions on $\mathbb{D}$ such that
\begin{align*}u(\lambda)=\inf_{h\in {\mathcal{H}}} h(\lambda) \ \ \, \lambda \in \mathbb{D}.
\end{align*}
\end{definition}

Note that an inf-sym-harmonic is required to be non-negative.
We will need the following variant of Harnack's inequality for such functions, stated as Lemma 10.7 in \cite{FRY}. See also Lemma 13.3.8 in \cite{AIM} for an equivalent version.
\begin{lemma}[Lemma 10.7 \cite{FRY}]
\label{Lemma107}
Let $v:\mathbb{D} \to [0, \infty)$ be an inf-sym-harmonic function. Then
\begin{align*}
\frac{1-y^2}{1+y^2}v(0)\leq v(iy)\leq \frac{1+y^2}{1-y^2}v(0),    \quad \text{for} \ y\in (-1, 1).
\end{align*}
\end{lemma}

\section{Dimension distortion for antisymmetric maps} \label{se:antisymmetric}

In this section, we establish dimension distortion results for antisymmetric maps. 
We make use of the recent work \cite{FRY} which leads to elegant short proofs. 
Let us introduce the following functions.
Given $t\in (0, 2]$, set
\begin{align}
\label{t}
t(k)=\frac{(1+k^2)t}{1-k^2+k^2t}\in (0, 2], \quad \text{and} \quad t^*(k)=\frac{(1-k^2)t}{1+k^2-k^2t}\in (0, 2]. \end{align}

We will start with the dimension estimate from above. 
\begin{theorem}
\label{upward}
Let $E\subset {\R}$ such that $0<t:=\dimh (E)\leq1$. Suppose that $f:{\C} \to {\C}$ is a $K$-quasiconformal antisymmetric map.  Then $\dimh (f(E))\leq t(k)$, where $t(k)$ is defined in \eqref{t}.
\end{theorem}
\begin{proof}
Without loss of generality we may assume that $f$ is normalised so that $f(0)=0$ and $f(1)=1$. 
By antisymmetric property of function $f$ we can find a symmetric holomorphic motion $\Phi:\mathbb{D} \times {\C} \to {\C}$ such that $\Phi_{ik}=f$. Here, symmetry of holomorphic motion means that $\Phi_{\lambda}(z)=\overline{\Phi_{\overline{\lambda}}(\overline{z})}$ for all $\lambda\in \mathbb{D}$ and $z\in {\C}$. 
The construction of such holomorphic motion is based on the standard proof of Theorem \ref{thm:embedded}, see e.g. \cite[page 479]{FRY}.
By Lemma 10.4 of \cite{FRY}, there is an inf-sym-harmonic function $u$ on $\mathbb{D}$ such that 
\begin{align}
\label{tahti1}
u(0)=\frac{1}{\dimh(E)}         \ \  \ \textrm{and} \ \ \  \frac{1}{2}\leq u(\lambda)\leq \frac{1}{\dimh(E_{\lambda})},
\end{align}
where $\lambda \in \mathbb{D}$ and $E_{\lambda}:=\Phi_{\lambda}(E)$. The function $v:=u-\frac{1}{2}$ is also inf-sym-harmonic and by Lemma \ref{Lemma107} 
\begin{align}
\label{tahti2}
v(ik)\geq \frac{1-k^2}{1+k^2}v(0)=\frac{1-k^2}{1+k^2}\left(\frac{1}{\dimh(E)}-\frac{1}{2}\right).
\end{align}
Using \eqref{tahti1}, we obtain
\begin{align}
\label{tahti3}
v(ik)=u(ik)-\frac{1}{2}\leq \frac{1}{\dimh(E_{ik})}-\frac{1}{2}=\frac{1}{\dimh(\Phi_{ik}(E))}-\frac{1}{2}=\frac{1}{\dimh(f(E))}-\frac{1}{2}.
\end{align}
By combining \eqref{tahti2} and \eqref{tahti3}, we obtain
\begin{align}
\label{tahti4}
\frac{1-k^2}{1+k^2}\left(\frac{1}{\dimh(E)}-\frac{1}{2}\right) \leq \frac{1}{\dimh(f(E))}-\frac{1}{2}
\end{align}
It follows from \eqref{tahti4} that
\begin{align*}
\dimh(f(E)) \leq \frac{(1+k^2)\dimh(E)}{1-k^2+k^2\dimh(E)}=\frac{(1+k^2)t}{1-k^2+k^2t}.
\end{align*}
\end{proof}

We now turn to the lower estimate for antisymmetric maps and first discuss (upper) Minkowski dimension, denoted by $\udimm$.

\begin{theorem} \label{thm:Minkowskilower}
Let $E\subset {\R}$ be a bounded set with $0<t:=\udimm (E)\leq1$. Suppose that $f:{\C} \to {\C}$ is antisymmetric $K$-quasiconformal  map.  Then $\udimm (f(E))\geq t^*(k)$, where $t^*(k)$ is defined in \eqref{t}.
\end{theorem}

\begin{proof}
We start in the same way as in the proof of Theorem \ref{upward}. By antisymmetric property of function $f$  we have a symmetric holomorphic motion $\Phi:\mathbb{D} \times {\C} \to {\C}$ such that $\Phi_{ik}=f$. By Lemma 10.3 of \cite{FRY}, there is an inf-sym-harmonic function $u$ on $\mathbb{D}$ such that 
\begin{align}
\label{tahti11}
u(0)=\frac{1}{\udimm(E)}         \ \  \ \textrm{and} \ \ \  u(\lambda)\geq \frac{1}{\udimm(E_{\lambda})},
\end{align}
where $\lambda \in \mathbb{D}$ and $E_{\lambda}:=\Phi_{\lambda}(E)$. The function $v:=u-\frac{1}{2}$ is also inf-sym-harmonic, since it is non-negative, and by Lemma \ref{Lemma107}
\begin{align}
\label{tahti22}
v(ik)\leq \frac{1+k^2}{1-k^2}v(0)=\frac{1+k^2}{1-k^2}\left(\frac{1}{\udimm(E)}-\frac{1}{2}\right).
\end{align}
Using \eqref{tahti11}, we obtain
\begin{align}
\label{tahti33}
v(ik)=u(ik)-\frac{1}{2}\geq \frac{1}{\udimm(E_{ik})}-\frac{1}{2}=\frac{1}{\udimm(\Phi_{ik}(E))}-\frac{1}{2}=\frac{1}{\udimm(f(E))}-\frac{1}{2}.
\end{align}
By combining \eqref{tahti22} and \eqref{tahti33}, we obtain
\begin{align}
\label{tahti44}
\frac{1+k^2}{1-k^2}\left(\frac{1}{\udimm(E)}-\frac{1}{2}\right) \geq \frac{1}{\udimm(f(E))}-\frac{1}{2}
\end{align}
It follows from \eqref{tahti44} that
\begin{align*}
\udimm(f(E)) \geq \frac{(1-k^2)\udimm(E)}{1+k^2-k^2\udimm(E)}=\frac{(1-k^2)t}{1+k^2-k^2t}.
\end{align*}
\end{proof}

A shortcoming in the argument above is that it only works for Minkowski dimension. Since we are going to need the result for Hausdorff dimension, we give the proof for Hausdorff dimension, too, see Theorem \ref{downward} below. The proof is based on Section 13 of \cite{AIM} and we will first need to establish a number of auxiliary results.

We start with an extension of Theorem 13.3.6 from \cite{AIM}.
\begin{theorem}
\label{AIM13.3.6vastineALAS}
Let $f:{\C} \to {\C}$ be an antisymmetric $K$-quasiconformal map for which
\begin{itemize}
	\item $f$ is a principal mapping, conformal outside the unit disk $\mathbb{D}$. \\
	\item There is a finite union of disjoint disks $E=B(z_1, r_1)\cup \cdots \cup B(z_n, r_n) \subset \mathbb{D}$ and $z_1, ..., z_n\in {\R}$, on which $f$ is also conformal. 
\end{itemize}
We have the upper bound
\begin{align*}
\sum_{j=1}^n \left(\left|f'(z_j)\right|r_j\right)^{t(k)}\leq 8^{t(k)} \left(\sum_{j=1}^n r_j^t\right)^{\frac{1-k^2}{1+k^2} \frac{t(k)}{t}},
\end{align*}
and the lower bound
\begin{align*}
\sum_{j=1}^n \left(\left|f'(z_j)\right|r_j\right)^{t^*(k)}\geq 8^{t^*(k)} 8^{-\frac{1+k^2}{1-k^2} t^*(k)} \left(\sum_{j=1}^n r_j^t\right)^{\frac{1+k^2}{1-k^2} \frac{t^*(k)}{t}},
\end{align*}
with the exponents $t(k)$ and $t^*(k)$ defined in \eqref{t}.
\end{theorem}

\old{Alaspain arvioksi saadaan (jokin parempi valiteksti)
\begin{theorem}
\label{AIM13.3.6vastineALAS}
Let $f:{\C} \to {\C}$ be an  antisymmetric $K$-quasiconformal map for which
\begin{itemize}
	\item $f$ is a principal mapping, conformal outside the unit disk $\mathbb{D}$. \\
	\item There is a finite union of disjoint disks $E=B(z_1, r_1)\cup \cdots \cup B(z_n, r_n) \subset \mathbb{D}$ and $z_1, ..., z_n\in {\R}$, on which $f$ is also conformal. 
\end{itemize}
Given $t\in (0, 2]$, set
\begin{align}
\label{ttahti}
t^*(k)=\frac{(1-k^2)t}{1+k^2-k^2t}\in (0, 2].
\end{align}
Then
\begin{align*}
\sum_{j=1}^n \left(\left|f'(z_j)\right|r_j\right)^{t^*(k)}\geq 8^{t^*(k)} 8^{-\frac{1+k^2}{1-k^2} t^*(k)} \left(\sum_{j=1}^n r_j^t\right)^{\frac{1+k^2}{1-k^2} \frac{t^*(k)}{t}}.
\end{align*}
\end{theorem}}

\begin{proof}
The upper bound is in fact the statement of \cite[Theorem 13.3.6]{AIM}. The proof of the lower bound is similar to that of the upper bound. We only need minor modifications. For the lower bound we use the left-hand side inequality of \cite[Lemma 13.3.8]{AIM} instead of the right-hand side. This then results in the change of the exponent $t^*(k)$ compared to $t(k)$. The rest of the proof is the same as for the upper bound.
\end{proof}

Our next goal is to remove the conformality assumptions in the disks $B(z_i, r_i)$. This can be done by following a deformation argument in \cite[p. 341]{AIM} which we review.
Let $f:{\C}\to {\C}$ be an antisymmetric $K$-quasiconformal. For $\lambda\in \mathbb{D}$, set $\mu_{\lambda}=\frac{\lambda}{\alpha}\mu_f$, where $k<\alpha<1$. Then for $f_{\lambda}$, the principal solution to the Beltrami equation
\begin{align*}
\frac{\partial f_{\lambda}}{\partial\overline{z}}(z)=\mu_{\lambda}(z)\frac{\partial f_{\lambda}}{\partial z}(z) 
\end{align*}
we have $f_{\alpha}=f$. Note that $\left\|\mu_\lambda \right\|_{\infty}\leq \frac{k}{\alpha}<1$, so $f_{\lambda}$ is $\frac{1+\frac{k}{\alpha}}{1-\frac{k}{\alpha}}$-quasiconformal independent of $\lambda$. Next, take a family of disjoint disks $E=B(z_1, r_1)\cup \cdots \cup B(z_n, r_n) \subset \mathbb{D}$ such that $z_1, ..., z_n\in {\R}$. As in \cite[corollary 12.7.2]{AIM}, we deform $f_{\lambda}$ on these disk $B(z_j, r_j)$, while retaining its value outside $E$. We obtain $\widetilde{f_{\lambda}}:{\C} \to {\C}$ a $\frac{1+\left|\lambda\right|}{1-\left|\lambda\right|}$-quasiconformal mapping such that
\begin{itemize}
	\item $\widetilde{f_{\lambda}}|_{{\C}\setminus E}=f_{\lambda}$ \\
	\item $\widetilde{f_{\lambda}}(z_j+r_j)=f_{\lambda}(z_j+r_j),\  \widetilde{f_{\lambda}}(z_j)=f_{\lambda}(z_j)$ \\
	\item $\widetilde{f_{\lambda}}|_{B(z_j, \ell r_j)}$ is a similarity, where $\ell= e^{-\pi\frac{1+\frac{k}{\alpha}}{1-\frac{k}{\alpha}}}$\\
	\item $(\widetilde{f_{\lambda}})'(z_j)=\frac{f_{\lambda}(z_j+r_j)-f_{\lambda}(z_j)}{r_j}$
\end{itemize}
Applying the upper bound of Theorem \ref{AIM13.3.6vastineALAS} to this deformed family we obtain, see equation (13.50) in \cite{AIM}.

\begin{proposition} 
\label{prop:AIM13.3.6}
Let $f:{\C} \to {\C}$ be an antisymmetric $K$-quasiconformal principal mapping such that $f$ is conformal outside the unit disk $\mathbb{D}$. Let $E=B(z_1, r_1)\cup \cdots \cup B(z_n, r_n) \subset \mathbb{D}$ be a family of disjoint disks such that $z_1, ..., z_n\in {\R}$. Then 
\begin{align*}
\sum_{j=1}^n \left(\left|(\widetilde{f_{\alpha}})'(z_j)\right|r_j\right)^{t(\alpha)}\leq \left(\frac{8}{\ell}\right)^{t(\alpha)} \left(\sum_{j=1}^n r_j^t\right)^{\frac{1-{\alpha}^2}{1+{\alpha}^2} \frac{t(\alpha)}{t}}, 
\end{align*}
where $k<\alpha<1$, $\ell=e^{-\pi\frac{1+\frac{k}{\alpha}}{1-\frac{k}{\alpha}}}$ and function $t$ is as in \eqref{t}. 
\end{proposition}
In a very similar way, using the lower bound in Theorem \ref{AIM13.3.6vastineALAS} we obtain.
\begin{proposition}
\label{prop:alas}
Let $f:{\C} \to {\C}$ be an antisymmetric $K$-quasiconformal principal mapping such that $f$ is conformal outside the unit disk $\mathbb{D}$. Let $E=B(z_1, r_1)\cup \cdots \cup B(z_n, r_n) \subset \mathbb{D}$ be a family of disjoint disks such that $z_1, ..., z_n\in {\R}$. Then 
\begin{align}
\label{prop:alas:ensimmainen kaavarivi}
\sum_{j=1}^n \left(\left|\widetilde{(f_{\alpha}})'(z_j)\right|r_j\right)^{t^*(\alpha)}\geq 8^{t^*(\alpha)} \cdot {\ell}^{-t^*(\alpha)} \cdot 8^{-\frac{1+{\alpha}^2}{1-{\alpha}^2} t^*(\alpha)} \cdot {\ell}^{\frac{1+{\alpha}^2}{1-{\alpha}^2}t^*(\alpha)}\cdot \left(\sum_{j=1}^n r_j^t\right)^{\frac{1+{\alpha}^2}{1-{\alpha}^2} \frac{t^*(\alpha)}{t}}
\end{align}
and further
\begin{align}
\label{prop:alas:toinen kaavarivi}
& \sum_{j=1}^n \left(\diam f(B(z_j, r_j))\right)^{t^*(\alpha)} \geq \nonumber \\
&8^{t^*(\alpha)} \cdot {\ell}^{-t^*(\alpha)} \cdot 8^{-\frac{1+{\alpha}^2}{1-{\alpha}^2} t^*(\alpha)} \cdot {\ell}^{\frac{1+{\alpha}^2}{1-{\alpha}^2}t^*(\alpha)}\cdot \left(\sum_{j=1}^n r_j^t\right)^{\frac{1+{\alpha}^2}{1-{\alpha}^2} \frac{t^*(\alpha)}{t}},
\end{align}
where $k<\alpha<1$, $\ell=e^{-\pi\frac{1+\frac{k}{\alpha}}{1-\frac{k}{\alpha}}}$ and function $t^*$ is as in \eqref{t}. 
\end{proposition}
\begin{proof}
The claim \eqref{prop:alas:ensimmainen kaavarivi} follows by applying the lower bound of Theorem \ref{AIM13.3.6vastineALAS} to deformed family $\widetilde{f_{\lambda}}$. The second claim follows by combining inequality
\begin{align*}
\diam f(B(z_j, r_j))\geq \left|f(z_j+r_j)-f(z_j)\right|=\left|(\widetilde{f_{\alpha}})'(z_j)\right|r_j
\end{align*}
and inequality \eqref{prop:alas:ensimmainen kaavarivi}.
\end{proof}

With these preparations we are ready to prove the lower estimate for Hausdorff dimension.

\begin{theorem}
\label{downward}
Let $E\subset {\R}$ be a set with $0<t:=\dimh (E)\leq1$. Suppose that $f:{\C} \to {\C}$ is antisymmetric $K$-quasiconformal  map.  Then $\dimh (f(E))\geq t^*(k)$, with $t^*(k)$ defined in \eqref{t}.
\end{theorem}
\begin{proof}
As the set $E$ can be covered by countably many sets of the form $E \cap [n/2,n/2+1]$, $n \in \mathbb{Z}$, without loss of generality we may assume that $E\subset [-\frac{1}{2}, \frac{1}{2}]$ with $0<t:=\dimh (E)\leq1$. We may also further assume that $f$ is a principal mapping. Indeed by setting the Beltrami coefficient $\mu_{\tilde f}:=\mu_f \cdot \chi_{\mathbb{D}}$, we have a decomposition $f=h \circ \tilde f$, where $\tilde f$ is a $K$-quasiconformal principal mapping and $h \colon \C \to \C$ is a $K$-quasiconformal map which is conformal on $\tilde f(\mathbb{D})$. Since $h$ is conformal on $\tilde f(\mathbb{D})$ the Hausdorff dimension is unaffected.

Let $\epsilon>0$ such that $t-\epsilon>0$. Then $\mathcal{H}^{t-\epsilon}(E)=\infty$. Hence there exist $M\geq 2^{t-\epsilon}>0$ such that for all $0<\delta\leq \delta_M$ holds 
\begin{align}
\label{deltaM}
\mathcal{H}_{\delta}^{t-\epsilon}(E)\geq M. 
\end{align}
Let $0<\delta<\delta'$, where we will choose $\delta'$ to be suitably small. 
Further, let $B_1:=B(x_1, r_1), ..., B_n:=B(x_n, r_n)$ be  a collection of balls such that $\diam B_i\leq \delta$ for all $i\in\{1, ...,n\}$ and $f(E)\subset \cup_{i=1}^n B_i$. We can assume that the centers $x_i$ of the balls $B_i$ belong to $f(E)$ for all $i\in\{1, ..., n\}$, see Remark \ref{huomatusjohonviitataan}. Then $f^{-1}(x_i)\in E\subset [-\frac{1}{2}, \frac{1}{2}]$ for all $i\in\{1, ..., n\}$. Note that sets $f^{-1}(B_1), ..., f^{-1}(B_n)$ covers to set $E$. We define (for all $i\in\{1, ...,n\}$) two balls: an circumscribed and inscribed ball. Let $B_i^{'}=\overline{B}(f^{-1}(x_i), r_{b_i})$ be the smallest closed ball centered at $f^{-1}(x_i)$ which contains $f^{-1}(B_i)$ and $B_i^{''}=B(f^{-1}(x_i), r_{s_i})$ be the biggest open ball centered at $f^{-1}(x_i)$ which is contained in $f^{-1}(B_i)$. By the well-known quasisymmetry property of $f$ there exist constant (depending only on $K$) $C(K)$ such that 
\begin{align}
\label{CK}
1\leq \frac{r_{b_i}}{r_{s_i}}\leq C(K)
\end{align}
for all $i\in\{1, ...,n\}$. Observe, that the balls  $B_1^{'}, ..., B_n^{'}$ cover the set $E$ and their diameter is uniformly small depending on $\delta'$, since $f$ and $f^{-1}$ is H\"older continuous \cite[Corollary 3.10.3]{AIM}. We now choose $\delta'$ small enough such that the diameters above are smaller than $\frac{1}{5}\delta_M$. By the 5r-covering theorem there exist a subcollection of $B_1^{'}, ..., B_n^{'}$, say $B_{j_1}^{'}, ... B_{j_{n'}}^{'}$, such that balls in this subcollection are pairwise disjoint and
\begin{align*}
E\subset \cup_{i=1}^{n'} 5B_{j_i}^{'}.
\end{align*}
The balls $\{5B_{j_i}^{'}\}_{i=1}^{n'}$ form a $\delta_M$-cover of $E$. Thus by \eqref{deltaM} 
\begin{align}
\label{vahintaan}
\sum_{i=1}^{n'} {\left(\diam 5B_{j_i}'\right)}^{t-\epsilon}\geq M\geq 2^{t-\epsilon}.
\end{align}
We can write
\begin{align}
\label{yhtasuuri}
\sum_{i=1}^{n'} {\left(\diam 5B_{j_i}'\right)}^{t-\epsilon}=\sum_{i=1}^{n'} (2\cdot 5r_{b_{j_i}})^{t-\epsilon}=10^{t-\epsilon} \sum_{i=1}^{n'} (r_{b_{j_i}})^{t-\epsilon}
\end{align}
By combining \eqref{vahintaan} and \eqref{yhtasuuri} we obtain
\begin{align*}
10^{t-\epsilon} \sum_{i=1}^{n'} (r_{b_{j_i}})^{t-\epsilon}\geq 2^{t-\epsilon}.
\end{align*}
This can be written in forms
\begin{align}
\label{estimate15}
\sum_{i=1}^{n'} (r_{b_{j_i}})^{t-\epsilon}\geq \left(\frac{1}{5}\right)^{t-\epsilon}.
\end{align}
By combining \eqref{CK} and \eqref{estimate15} gives
\begin{align}
\label{smallballsvsbigballs}
\sum_{i=1}^{n'} (r_{s_{j_i}})^{t-\epsilon}\geq \sum_{i=1}^{n'} \left(r_{b_{j_i}}\frac{1}{C(K)}\right)^{t-\epsilon}=\left(\frac{1}{C(K)}\right)^{t-\epsilon} \sum_{i=1}^{n'} \left(r_{b_{j_i}}\right)^{t-\epsilon}\geq \left(\frac{1}{C(K)}\right)^{t-\epsilon} \left(\frac{1}{5}\right)^{t-\epsilon}.
\end{align}
Let $0\leq k<\alpha<1$ and $\ell=e^{-\pi \frac{1+k/\alpha}{1-k/\alpha}}$ as in the statement of Proposition \ref{prop:alas}. We obtain
\begin{align}
\label{pitkakaavarivi}
&\sum_{i=1}^n {\left( \diam B_i\right)}^{(t-\epsilon)^*(\alpha)}=\sum_{i=1}^n {\left( \diam f(f^{-1}(B_i))\right)}^{(t-\epsilon)^*(\alpha)}\nonumber\\
&\geq \sum_{i=1}^n {\left( \diam f(B_i^{''})\right)}^{(t-\epsilon)^*(\alpha)}\geq \sum_{i=1}^{n'} {\left( \diam f(B_{j_i}^{''})\right)}^{(t-\epsilon)^*(\alpha)}\nonumber\\
&\geq 8^{(t-\epsilon)^*(\alpha)} \cdot {\ell}^{-(t-\epsilon)^*(\alpha)}\cdot 8^{-\frac{1+{\alpha}^2}{1-{\alpha}^2}(t-\epsilon)^*(\alpha)}\cdot {\ell}^{\frac{1+{\alpha}^2}{1-{\alpha}^2}(t-\epsilon)^*(\alpha)}\cdot \left(\sum_{i=1}^{n'} (r_{s_{j_i}})^{t-\epsilon}\right)^{\frac{1+{\alpha}^2}{1-{\alpha}^2}\frac{(t-\epsilon)^*(\alpha)}{t-\epsilon}}\nonumber\\
&\geq 8^{(t-\epsilon)^*(\alpha)} \cdot {\ell}^{-(t-\epsilon)^*(\alpha)}\cdot 8^{-\frac{1+{\alpha}^2}{1-{\alpha}^2}(t-\epsilon)^*(\alpha)}\cdot {\ell}^{\frac{1+{\alpha}^2}{1-{\alpha}^2}(t-\epsilon)^*(\alpha)}\\
&\cdot\left(\left(\frac{1}{C(K)}\right)^{t-\epsilon} \left(\frac{1}{5}\right)^{t-\epsilon}\right)^{\frac{1+{\alpha}^2}{1-{\alpha}^2}\frac{(t-\epsilon)^*(\alpha)}{t-\epsilon}}:=D(t, \epsilon,\alpha, K)>0,\nonumber
\end{align}
where the third inequality follows from \eqref{prop:alas:toinen kaavarivi} of Proposition \ref{prop:alas} and the last inequality follows from \eqref{smallballsvsbigballs}.
Thus by \eqref{pitkakaavarivi}
\begin{align}
\label{rajankaynti}
{\mathcal{B}}^{(t-\epsilon)^*(\alpha)}(f(E))=\lim_{\delta \downarrow 0} {{\mathcal{B}}_{\delta}}^{(t-\epsilon)^*(\alpha)} (f(E))\geq \lim_{\delta \downarrow 0} D(t, \epsilon, \alpha, K)=D(t, \epsilon, \alpha, K)>0.
\end{align}
Further by \eqref{pallopeitteetvrtkaikki peitteet}
\begin{align}
\mathcal{H}^{(t-\epsilon)^*(\alpha)}(f(E))\geq \frac{{\mathcal{B}}^{(t-\epsilon)^*(\alpha)}(f(E))}{2^{(t-\epsilon)^*(\alpha)}}\geq \frac{D(t, \epsilon, \alpha, K)}{2^{(t-\epsilon)^*(\alpha)}}>0.
\end{align}
Hence
\begin{align*}
\dimh f(E)\geq (t-\epsilon)^*(\alpha).
\end{align*}
Letting $\epsilon \downarrow 0$, we have $\dimh f(E) \geq t^*(\alpha)$.
Finally, we let $\alpha$ tend to $k$ to obtain
\begin{align*}
\dimh f(E) \geq \lim_{\alpha \downarrow k} t^*(\alpha)=t^*(k).
\end{align*}
\end{proof}
\begin{remark}
\label{huomatusjohonviitataan}
In the argument we are interested in optimal $\delta$-cover of the set $f(E)$ by balls. Thus we can consider such $\delta$-covers of the set $f(E)$, where all balls intersect the set $f(E)$. As in proof of Theorem \ref{downward}, let $B_1:=B(x_1, r_1), ..., B_n:=B(x_n, r_n)$ be a collection of balls such that $\diam B_i\leq \delta$ for all $i\in\{1, ...,n\}$, $f(E)\cap B_i\neq \emptyset$ for all $i\in\{1, ...,n\}$ and $f(E)\subset \cup_{i=1}^n B_i$. We can assume that center points $x_i$ of balls $B_i$ are elements of $f(E)$ for all $i\in\{1, ..., n\}$. Indeed, if there exist some ball $B_i$ such that its center point $x_i\notin f(E)$, then we proceed as follows. We choose a point from the set $B_i\cap f(E)$ and draw $\diam B_i$-radius circle with this point as the center. Let this new $2\diam B_i$-diameter ball be $B_i^*$. In this way, the diameters of all balls can be doubled. If we choose $\delta'$ small enough, then an circumscribed and inscribed ball corresponding to $B_i^*$ satisfies \eqref{smallballsvsbigballs}) and we obtain for the balls $B_i^*$
\begin{align*}
&\sum_{i=1}^n {\left( \diam B_i\right)}^{(t-\epsilon)^*(\alpha)}=\frac{1}{2^{(t-\epsilon)^*(\alpha)}}\sum_{i=1}^n {\left( \diam B_i^*\right)}^{(t-\epsilon)^*(\alpha)}\\
&\geq...\geq \frac{1}{2^{(t-\epsilon)^*(\alpha)}}D(t, \epsilon, \alpha, K)>0.
\end{align*} 

\end{remark}

\section{Improved dimension estimates on the line}
\label{se:improved}

Let us recall that the bounds of Astala in \eqref{Astalasestimate} are sharp for subsets $E \subset \C$, see \cite{Astala}. Here we show that these bounds can be improved if we consider subsets of the line $E \subset \R$.

These improved estimates will be based on the decomposition property mentioned in Section 2. The precise statement is as follows, see Section 13.3.1 of \cite{AIM}. 

\begin{proposition}
A $K$-quasiconformal map $f^*:{\C}\to {\C}$ can be written $f^*=\widetilde{f}\circ h$, where $\widetilde{f}:{\C}\to {\C}$ is \emph{antisymmetric} $K$-quasiconformal  and $h:{\C}\to {\C}$ is a \emph{symmetric} $K$-quasiconformal map.
\end{proposition}

In Section 3 we established dimension distortion estimates for antisymmetric maps. For symmetric quasiconformal maps we will use Theorem 1.1 of \cite{PrSm} which gives bounds in terms of the function 
\[ \Delta(\dimh(E), k)=\frac{\dimh(E)(1-k^2)}{(1+k\sqrt{1-\dimh(E)})^2}.
\]

We can now combine these two type of estimates. Let $E \subset \R$, note that then also $h(E)\subset {\R}$. We start with the estimate from below. Combining Theorem \ref{downward} and \cite[Theorem 1.1]{PrSm}, we obtain
\begin{align}
\label{destimate}
\dimh(f^*(E))&= \dimh(\widetilde{f}(h(E)))\geq \frac{(1-k^2)\dimh(h(E))}{1+k^2-k^2\dimh(h(E))}\nonumber\\
&\geq \frac{(1-k^2)\Delta(\dimh(E), k)}{1+k^2-k^2\Delta(\dimh(E), k))}\nonumber\\
&=\frac{(1-k^2)\frac{\dimh(E)(1-k^2)}{(1+k\sqrt{1-\dimh(E)})^2}}{1+k^2-\frac{k^2\dimh(E)(1-k^2)}{(1+k\sqrt{1-\dimh(E)})^2}}.
\end{align}

Next, let us consider the estimate from above. Suppose first that $\dimh(E)\leq 1-k^2$. By combining Theorem \ref{upward} and \cite[Theorem 1.1]{PrSm}, we obtain
\begin{align}
\label{uestimate}
\dimh(f^*(E))&= \dimh(\widetilde{f}(h(E)))\leq \frac{(1+k^2)\dimh(h(E))}{1-k^2+k^2\dimh(h(E))}\nonumber\\
&\leq \frac{(1+k^2)\Delta^*(\dimh(E), k)}{1-k^2+k^2\Delta^*(\dimh(E), k))}\nonumber\\
&=\frac{(1+k^2)\dimh(E)}{1+k^2-2k\sqrt{1-\dimh(E)}},
\end{align}
where $\Delta^*(\dimh(E), k)=\Delta(\dimh(E), -\min\{k, \sqrt{1-\dimh(E)}\})$ as in the \cite[Theorem 1.1]{PrSm}. Note that assumption $\dimh(E)\leq 1-k^2$ means that 
\begin{align*}
\min\{k, \sqrt{1-\dimh(E)}\}=k
\end{align*}
and thus 
\begin{align*}
\Delta^*(\dimh(E), k)=\Delta(\dimh(E), -k)=\frac{\dimh(E)(1-k^2)}{(1-k\sqrt{1-\dimh(E)})^2}. 
\end{align*}

In the case $1-k^2<\dimh(E)\leq 1$, we can directly use the quasicircle bound of \cite{Smi} 
\begin{align}
\label{uestimate2}
\dimh(f^*(E))\leq 1+k^2. 
\end{align}

The estimates \eqref{destimate} and \eqref{uestimate} do not necessarily always give better estimates compared to Astala's original bounds \eqref{Astalasestimate}. However, they are better when the quasiconformal distortion $K$ is close to $1$. To make use of this fact, we introduce one more decomposition. Namely, an arbitrary $K$-quasiconformal map $f:{\C}\to {\C}$ can be decomposed as $f=f_1\circ f_2$, where $f_1:{\C}\to {\C}$ is $K_1$-quasiconformal and $f_2:{\C}\to {\C}$ is $K_2$-quasiconformal such that $K=K_1K_2$. 
This fact is well-known, see e.g. the proof of Theorem 5.6.1 in \cite{AIM}.
We will apply the estimates \eqref{destimate} and \eqref{uestimate} to $f_2$ and \eqref{Astalasestimate} to $f_1$. Our job is now to choose the value $K_2$ (depending on $K$ and $\dimh E$) suitably. In principle, one could optimise for this choice, but unfortunately the resulting expressions are overly complicated. Therefore we are content to present weaker versions below with concrete estimates in Theorems \ref{smallerthanone1} and \ref{smallerthanone2}.

We note that the estimates \eqref{Astalasestimate} have the property that they could be derived in two steps separately applying for $f_1$ and $f_2$ in the decomposition above. Since in our argument we use the same estimate for $f_1$ we only need to concentrate on the map $f_2$ and compare these two estimates.
For this reason we introduce a couple of auxiliary functions. We denote by $L=\dimh(E)$ and start with the estimate from below. We define the function $g_0$ as a difference between the two estimates appearing in \eqref{destimate} and \eqref{Astalasestimate} as applied for the map $f_2$
\begin{align*}
g_0(k_2, L):=\frac{(1-k_2^2)\frac{L(1-k_2^2)}{(1+k_2\sqrt{1-L)^2}}}{1+k_2^2-\frac{k_2^2L(1-k_2^2)}{(1+k_2\sqrt{1-L)^2}}}-\frac{1}{\frac{1+k_2}{1-k_2}(\frac{1}{L}-\frac{1}{2})+\frac{1}{2}}, \quad k_2 \in [0,k].
\end{align*}

Similarly, in the case of estimate from above, we define two more auxiliary functions as the relevant differences as follows
\begin{align*}
g_1(k_2, L):=\frac{1}{\frac{1-k_2}{1+k_2}(\frac{1}{L}-\frac{1}{2})+\frac{1}{2}}-\frac{(1+k_2^2)L}{1+k_2^2-2k_2\sqrt{1-L}}, \quad k_2 \in [0, \min\{k, \sqrt{1-L}\}],
\end{align*}
and
\begin{align*}
g_2(k_2, L):=\frac{1}{\frac{1-k_2}{1+k_2}(\frac{1}{L}-\frac{1}{2})+\frac{1}{2}}-(1+k_2^2), \quad k_2 \in (\min\{k, \sqrt{1-L}\}, k].
\end{align*}

The following theorem improves on the bound in \eqref{Astalasestimate} from below when $E \subset \R$. 
\begin{theorem}\label{smallerthanone1}
Let $E\subset {\R}$ such that $0<L:=\dimh(E)<1$, $k\geq 1.5\cdot 10^{-12}$ and $f:{\C} \to {\C}$ be a $K$-quasiconformal map. 
\begin{enumerate}
\item If $0<L\leq 0.635212...$, then $g_0(L^{60}, L)>0$ and
\begin{align*}
\frac{1}{K\left(\frac{1}{L}-\frac{1}{2}\right)+\frac{1}{2}}<\frac{1}{\frac{K}{\frac{1+L^{60}}{1-L^{60}}}\left(\frac{1}{\frac{(1-L^{120})\frac{L(1-L^{120})}{(1+L^{60}\sqrt{1-L})^2}}{1+L^{120}-\frac{L^{121}(1-L^{120})}{(1+L^{60}\sqrt{1-L})^2}}}-\frac{1}{2}\right)+\frac{1}{2}}\leq\dimh(f(E)).\\
\end{align*} 
\item If $0.635212...<L<1$, then $g_0((1-L)^{27}, L)>0$ and 
\begin{align*}
\frac{1}{K\left(\frac{1}{L}-\frac{1}{2}\right)+\frac{1}{2}}<\frac{1}{\frac{K}{\frac{1+(1-L)^{27}}{1-(1-L)^{27}}}\left(\frac{1}{\frac{(1-(1-L)^{54})\frac{L(1-(1-L)^{54})}{(1+(1-L)^{27}\sqrt{1-L})^2}}{1+(1-L)^{54}-\frac{(1-L)^{54}L(1-(1-L)^{54})}{(1+(1-L)^{27}\sqrt{1-L})^2}}}-\frac{1}{2}\right)+\frac{1}{2}}\leq\dimh(f(E)).\\
\end{align*}
\end{enumerate}
\end{theorem}
\begin{proof}

Recall that a $K$-quasiconformal map $f:{\C}\to {\C}$ can be written as $f=f_1\circ f_2$, where $f_1:{\C}\to {\C}$ is $K_1$-quasiconformal and $f_2:{\C}\to {\C}$ is $K_2$-quasiconformal such that $K=K_1K_2$. We will choose $K_2=\frac{1+k_2}{1-k_2}$ such that $g_0(k_2, L)>0$. Note that then $K_1=\frac{K}{K_2}$ and
\begin{align*}
k_1=\frac{K_1-1}{K_1+1}=\frac{\frac{K}{K_2}-1}{\frac{K}{K_2}+1}=\frac{\frac{K}{\frac{1+k_2}{1-k_2}}-1}{\frac{K}{\frac{1+k_2}{1-k_2}}+1}.
\end{align*}
Since $E\subset {\R}$, we can use estimate \eqref{destimate} to the map $f_2$. First, we use right hand estimate of \eqref{Astalasestimate} for map $f_1$ and after that we use estimate \eqref{destimate} to map $f_2$. 
Note that after using right hand estimate of \eqref{Astalasestimate} for function $f_1$, we have strictly increasing function on interval $(0, 2]$ of the form
\begin{align*}
x\mapsto \frac{1}{\frac{1+k_1}{1-k_1}(\frac{1}{x}-\frac{1}{2})+\frac{1}{2}}.
\end{align*}
Therefore, it is essential to give a better estimate for $\dimh (f_2(E))$ from below than what right hand estimate of \eqref{Astalasestimate} gives. If $0<L\leq 0.635212...$, then we choose $k_2=L^{60}<1.5\cdot 10^{-12}\leq k$ and obtain $g_0(L^{60}, L)>0$. This follows from the numerical estimate
$g_0(x^{60}, x)>0$ valid in the range $0<x<0.986...$.
Hence
\begin{align}
\label{alas1}
& \frac{1}{\frac{1+k}{1-k}\left(\frac{1}{\dimh(E)}-\frac{1}{2}\right)+\frac{1}{2}}<\left( {\frac{1+k_1}{1-k_1}\left(\frac{1}{\frac{(1-k_2^2)\frac{\dimh(E)(1-k_2^2)}{(1+k_2\sqrt{1-\dimh(E)})^2}}{1+k_2^2-\frac{k_2^2\dimh(E)(1-k_2^2)}{(1+k_2\sqrt{1-\dimh(E)})^2}}}-\frac{1}{2}\right)+\frac{1}{2}} \right)^{-1}\\
&\leq\frac{1}{\frac{1+k_1}{1-k_1}\left(\frac{1}{\dimh(f_2(E))}-\frac{1}{2}\right)+\frac{1}{2}}\nonumber
\leq\dimh(f_1(f_2(E))=\dimh(f(E)) \nonumber.
\end{align}
Note that left-hand side presentation of \eqref{alas1} is Astala's original estimate for function $f$.
The strict inequality follows from the fact that $g_0(k_2, L)=g_0(L^{60}, L)>0$. 
The claim (1) now follows by substituting $k_2=L^{60}$ and $\frac{1+k_1}{1-k_1}=K_1=\frac{K}{K_2}=\frac{K}{\frac{1+L^{60}}{1-L^{60}}}$ in \eqref{alas1}.\\

If $0.635212...<L<1$, then we choose $k_2=(1-L)^{27}<1.5\cdot 10^{-12}\leq k$ and obtain $g_0((1-L)^{27}, L)>0$. Here we use the numerical estimate $g_0((1-x)^{27}, x)>0$ valid for $ 0.179...<x<1$. Note that the value $x_0=0.635212\ldots$ is the unique positive solution of the equation $x_0^{60}=(1-x_0)^{27}$. 
Hence we can write
\begin{align}
\label{alas2}
& \frac{1}{\frac{1+k}{1-k}\left(\frac{1}{\dimh(E)}-\frac{1}{2}\right)+\frac{1}{2}}<\left( {\frac{1+k_1}{1-k_1}\left(\frac{1}{\frac{(1-k_2^2)\frac{\dimh(E)(1-k_2^2)}{(1+k_2\sqrt{1-\dimh(E)})^2}}{1+k_2^2-\frac{k_2^2\dimh(E)(1-k_2^2)}{(1+k_2\sqrt{1-\dimh(E)})^2}}}-\frac{1}{2}\right)+\frac{1}{2}} \right)^{-1} \\
&\leq\frac{1}{\frac{1+k_1}{1-k_1}\left(\frac{1}{\dimh(f_2(E))}-\frac{1}{2}\right)+\frac{1}{2}}\nonumber 
\leq\dimh(f_1(f_2(E))=\dimh(f(E))\nonumber,
\end{align}
where the left-hand side inequality is a consequence of $g_0(k_2, L)=g_0((1-L)^{27}, L)>0$. Finally, the claim (2) follows by substituting $k_2=(1-L)^{27}$ and 
\[ \frac{1+k_1}{1-k_1}=K_1=\frac{K}{K_2}=\frac{K}{\frac{1+(1-L)^{27}}{1-(1-L)^{27}}}\]
in \eqref{alas2}.
\end{proof}

We now turn to estimates from above. Our result reads as the following in this setting.
\begin{theorem}\label{smallerthanone2}
Let $E\subset {\R}$ such that $0<L:=\dimh(E)<1$, $k\geq 2.67\cdot 10^{-21}$ and $f:{\C}\to {\C}$ be a $K$-quasiconformal map. 
\begin{enumerate}
\item If $0<L\leq 0.6197...$, then $g_1(L^{99}, L)>0$ and
\begin{align*}
\dimh(f(E))\leq \frac{1}{\frac{\frac{1+L^{99}}{1-L^{99}}}{K}\left(\frac{1}{ \frac{(1+L^{198})L}{1+L^{198}-2 L^{99}\sqrt{1-L}}}-\frac{1}{2}\right)+\frac{1}{2}}<\frac{1}{\frac{1}{K}\left(\frac{1}{L}-\frac{1}{2}\right)+\frac{1}{2}}.
\end{align*}
\item If $0.6197...<L\leq 1-(2.67)^2 \cdot 10^{-42}$, then $g_1((1-L)^{49}, L)>0$ and
\begin{align*}
\dimh(f(E))\leq \frac{1}{\frac{\frac{1+(1-L)^{49}}{1-(1-L)^{49}}}{K}\left(\frac{1}{ \frac{(1+(1-L)^{98})L}{1+(1-L)^{98}-2(1-L)^{49}\sqrt{1-L}}}-\frac{1}{2}\right)+\frac{1}{2}}<\frac{1}{\frac{1}{K}\left(\frac{1}{L}-\frac{1}{2}\right)+\frac{1}{2}}.
\end{align*}
\item If $1-(2.67)^2 \cdot 10^{-42}<L<1$, then $g_2(2.67\cdot 10^{-21}, L)>0$ and
\begin{align*}
\dimh(f(E))\leq \frac{1}{\frac{1-k_1}{1+k_1}\left(\frac{1}{1+k_2^2}-\frac{1}{2}\right)+\frac{1}{2}}<\frac{1}{\frac{1-k}{1+k}\left(\frac{1}{L}-\frac{1}{2}\right)+\frac{1}{2}},
\end{align*}
where $k_2=2.67\cdot 10^{-21}$ and $k_1=\frac{\frac{K}{\frac{1+k_2}{1-k_2}}-1}{\frac{K}{\frac{1+k_2}{1-k_2}}+1}$.
\end{enumerate}
\end{theorem}
\begin{proof}
We start in the same way as in the proof of Theorem \ref{smallerthanone1}. A $K$-quasiconformal map $f:{\C}\to {\C}$ can be decomposed as $f=f_1\circ f_2$, where $f_1:{\C}\to {\C}$ is $K_1$-quasiconformal and $f_2:{\C}\to {\C}$ is $K_2$-quasiconformal such that $K=K_1K_2$. In case (1) we will choose $K_2=\frac{1+k_2}{1-k_2}$ such that $g_1(k_2, L)>0$ . In case (2) we proceed similarly to that in case (1). In case (3) we will choose $K_2=\frac{1+k_2}{1-k_2}$ such that $g_2(k_2, L)>0$. Note that after fixing $K_2$, we have $K_1=\frac{K}{K_2}$ and
\begin{align*}
k_1=\frac{K_1-1}{K_1+1}=\frac{\frac{K}{K_2}-1}{\frac{K}{K_2}+1}=\frac{\frac{K}{\frac{1+k_2}{1-k_2}}-1}{\frac{K}{\frac{1+k_2}{1-k_2}}+1}.
\end{align*}
Since $E\subset {\R}$, we can use estimates \eqref{uestimate} and \eqref{uestimate2} to the map $f_2$. First, we use left-hand estimate of \eqref{Astalasestimate} to map $f_1$ and after that we use estimates \eqref{uestimate} and \eqref{uestimate2} to map $f_2$. 
Note that after using left-hand estimate of \eqref{Astalasestimate} for map $f_1$, we have strictly increasing function on interval $(0, 2]$ of the form
\begin{align*}
x\mapsto \frac{1}{\frac{1-k_1}{1+k_1}(\frac{1}{x}-\frac{1}{2})+\frac{1}{2}}.
\end{align*}

This time, we will have to give better estimate for $\dimh (f_2(E))$ from above than the left-hand estimate of \eqref{Astalasestimate}.
Note that, if $0<L\leq 1-(2.67)^2\cdot 10^{-42}$, then $\min \{k, \sqrt{1-L}\}\geq 2.67\cdot 10^{-21}$ and we can consider function $g_1$ on interval $[0, 2.67\cdot 10^{-21}]$ . If $0<L\leq 0.6197...$, then we choose $k_2=L^{99}<2.67\cdot 10^{-21}\leq k$ and obtain $g_1(L^{99}, L)>0$. 
For this we use the numerical estimate $g_1(x^{99}, x)>0$ valid for every $0<x\leq 1$.
Hence
\begin{align}
\label{ylos1}
& \dimh(f(E))=\dimh(f_1(f_2(E))\leq \frac{1}{\frac{1-k_1}{1+k_1}\left(\frac{1}{\dimh(f_2(E))}-\frac{1}{2}\right)+\frac{1}{2}} \\
&\leq \frac{1}{\frac{1-k_1}{1+k_1}\left(\frac{1}{\frac{(1+k_2^2)\dimh(E)}{1+k_2^2-2k_2\sqrt{1-\dimh(E)}}}-\frac{1}{2}\right)+\frac{1}{2}}
<\frac{1}{\frac{1-k}{1+k}\left(\frac{1}{\dimh(E)}-\frac{1}{2}\right)+\frac{1}{2}}\nonumber,
\end{align}
where last inequality follows from fact that $g_1(k_2, L)=g_1(L^{99}, L)>0$. 
Again, note that right-hand side presentation of \eqref{ylos1}) is Astala's original estimate for the map $f$.
The claim (1) follows by substituting $k_2=L^{99}$ and $\frac{1-k_1}{1+k_1}=\frac{K_2}{K}=\frac{\frac{1+L^{99}}{1-L^{99}}}{K}$ in \eqref{ylos1}.\\

If $0.6197...<L\leq 1-(2.67)^2\cdot 10^{-42}$, then we choose $k_2=(1-L)^{49}<2.67\cdot 10^{-21}\leq k$ and obtain $g_1((1-L)^{49}, L)>0$. For this we use the numerical estimate
$g_1((1-x)^{49}, x)>0$ valid for $0.119...<x< 1$. The role of the value $x_0=0.6197...$ this time is that it is the unique positive solution of the equation $x_0^{99}=(1-x_0)^{49}$. Hence we have
\begin{align}
\label{ylos2}
& \dimh(f(E))=\dimh(f_1(f_2(E))\leq \frac{1}{\frac{1-k_1}{1+k_1}\left(\frac{1}{\dimh(f_2(E))}-\frac{1}{2}\right)+\frac{1}{2}} \\
&\leq \frac{1}{\frac{1-k_1}{1+k_1}\left(\frac{1}{\frac{(1+k_2^2)\dimh(E)}{1+k_2^2-2k_2\sqrt{1-\dimh(E)}}}-\frac{1}{2}\right)+\frac{1}{2}}
<\frac{1}{\frac{1-k}{1+k}\left(\frac{1}{\dimh(E)}-\frac{1}{2}\right)+\frac{1}{2}}\nonumber,
\end{align}
where last inequality follows from fact that $g_1(k_2, L)=g_1((1-L)^{49}, L)>0$. The claim (2) follows by substituting $k_2=(1-L)^{49}$ and $\frac{1-k_1}{1+k_1}=\frac{K_2}{K}=\frac{\frac{1+(1-L)^{49}}{1-(1-L)^{49}}}{K}$ in \eqref{ylos2}.\\

If $1-(2.67)^2\cdot 10^{-42}<L<1$, then $\min \{k, \sqrt{1-L}\}=\sqrt{1-L}<2.67\cdot  10^{-21}$. Thus we have to consider function $g_2$ at the value $2.67\cdot 10^{-21}$. We choose $k_2=2.67\cdot 10^{-21}$ and obtain
\begin{align*}
g_2(2.67\cdot 10^{-21}, L)\geq g_2(2.67\cdot 10^{-21}, 1-10^{-40})\approx 2.67 \cdot 10^{-21}>0.
\end{align*}
In this case, we can estimate as follows
\begin{align*}
\dimh(f(E))&=\dimh(f_1(f_2(E)) \\
\leq \frac{1}{\frac{1-k_1}{1+k_1}\left(\frac{1}{\dimh(f_2(E))}-\frac{1}{2}\right)+\frac{1}{2}} &\leq \frac{1}{\frac{1-k_1}{1+k_1}\left(\frac{1}{1+k_2^2}-\frac{1}{2}\right)+\frac{1}{2}}<\frac{1}{\frac{1-k}{1+k}\left(\frac{1}{\dimh(E)}-\frac{1}{2}\right)+\frac{1}{2}},
\end{align*}
where $k_2=2.67\cdot 10^{-21}$. This proves claim (3).
\end{proof}

\end{document}